%% file: gibbard_trivalence_main.tex
\documentclass[11pt,a4paper]{article}
\usepackage[utf8]{inputenc}
\usepackage[T1]{fontenc}
\usepackage{verbatim} 
\usepackage[section]{placeins}

\usepackage{amsmath,enumitem}
\usepackage{amssymb}
\usepackage{amsthm}
			\newtheorem{thm}{Theorem}[]
			
			\newtheorem{lem}[thm]{Lemma}

\usepackage{url,caption} 
\usepackage[sc,osf]{mathpazo}
\usepackage{enumitem,textcomp}
\usepackage{linguex}
\usepackage[english]{babel}
\usepackage{geometry}
\usepackage{pifont}
\usepackage{proof} 
\usepackage{txfonts,pxfonts} 
\usepackage{gb4e} 
\noautomath 

\newcommand{\amrand}[2]
  {\leavevmode
   \marginpar
     [\raggedleft\scriptsize \leavevmode\normalcolor #1: #2]
     {\raggedright\scriptsize \leavevmode\normalcolor #1: #2}
  }

\usepackage{rotating}

\usepackage{tikz}

\usepackage{units}
\usetikzlibrary{arrows,decorations.pathmorphing,
backgrounds,positioning,fit,matrix}
\usepackage{tikz-qtree}
\usepackage{tikz-qtree-compat,textcomp}
\usetikzlibrary{calc}
\usetikzlibrary{backgrounds}
\usetikzlibrary{shapes.geometric}
\usetikzlibrary{shapes.arrows}
\usepackage{units}
\usepackage{booktabs} 
\usepackage{graphicx}
\usepackage{bussproofs} 
\usepackage{proof} 
\usepackage{eso-pic}

\newcommand{\scc}{\to}
\newcommand{\wcc}{\rightsquigarrow}

\usepackage{color, colortbl}
\definecolor{Gray}{gray}{.85}
\definecolor{Blue}{RGB}{60,20,255}
\definecolor{Red}{RGB}{255,20,60}

\usepackage{hyperref}
\hypersetup{colorlinks=true,allcolors=blue}

\usepackage[sectionbib,longnamesfirst]{natbib}
\setcitestyle{aysep={}} 
\shortcites{cobreros2015vagueness,Cobreros2012tolerant,OverEtAl2007,EvansEtAl2007}

\newcommand{\cctt}[0]{{\sf CC/TT}}
\newcommand{\dftt}[0]{{\sf DF/TT}}
\newcommand{\qcctt}[0]{{\sf QCC/TT}}
\newcommand{\half}{\nicefrac{1}{2}}  

\newcommand{\cmark}{\ding{51}}%
\newcommand{\xmark}{\ding{55}}%

\newcommand{\bim}{\supset\!\subset}

\title{Gibbardian Collapse and Trivalent Conditionals
} 

\author{Paul \'Egr\'e\footnote{Institut Jean-Nicod (CNRS/ENS/EHESS), D\'epartement de philosophie \& D\'epartement d'\'etudes cognitives, Ecole normale sup\'erieure, PSL University, 29 rue d'Ulm, 75005, Paris, France. ORCID: 0000-0002-9114-7686. Email: paul.egre@ens.fr} \and Lorenzo Rossi\footnote{Munich Center for Mathematical Philosophy (MCMP), Fakult\"{a}t f\"{u}r Philosophie, Wissenschaftstheorie 
und Religionswissenschaft, Ludwig-Maximilians-Universit\"{a}t M\"{u}nchen, Geschwister-Scholl-Platz 1, D-80539 München. ORCID: 0000-0002-1932-5484. Email: lorenzo.rossi@lrz.uni-muenchen.de} \and Jan Sprenger\footnote{Center for Logic, Language and Cognition (LLC), Department of Philosophy and Educational Science, Universit\`a degli Studi di Torino, Via Sant'Ottavio 20, 10124 Torino, Italy. ORCID: 0000-0003-0083-9685. Email: jan.sprenger@unito.it}}

\date{}

\begin{document}

\maketitle

\begin{abstract}
\noindent This paper discusses the scope and significance of the so-called triviality result stated by Allan Gibbard for indicative conditionals,
showing that if a conditional operator satisfies the Law of
Import-Export, is supraclassical, and is stronger than the material
conditional, then it must collapse to the material conditional.
Gibbard's result is taken to pose a dilemma for a truth-functional
account of indicative conditionals: give up Import-Export, or embrace the
two-valued analysis. We show that this dilemma can be averted in
trivalent logics of the conditional based on Reichenbach and de
Finetti's idea that a conditional with a false antecedent is
undefined. Import-Export and truth-functionality hold without
triviality in such logics. We unravel some implicit assumptions in
Gibbard's proof, and discuss a recent generalization of Gibbard's
result due to Branden Fitelson.\end{abstract}




\input{Intro}

\input{Gibbard5}

\input{TrivalentSK2}

\input{TrivalentQ}

\input{Fitelson}

\section{Conclusion}\label{sec:ccl}

This paper has given a precise reconstruction of Gibbard's informal argument that any indicative conditional that satisfies Import-Export and is supraclassical and stronger than the material conditional must collapse to the material conditional. Specifically, we have seen that Gibbard's argument requires additional premises (e.g., structural assumptions on the underlying logic $L$) and that the premises are not tight either (e.g., supraclassicality can be replaced without loss of validity by Conjunction Elimination). 

We have then explored how a family of trivalent logics, all based on the idea that a conditional is void when its antecedent turns out false, fare with respect to Gibbardian collapse. The logics we have examined all block an important premise of Gibbard's proof, namely the classical behavior of the material conditional $\supset$, as well as one additional premise (different for each logic). Nonetheless, in {\sf DF/TT}---the tolerant-to-tolerant logic based on de Finetti's truth table for the indicative conditional---Gibbardian collapse occurs, but this does not mean that both conditionals obey the same logical principles. In contrast, Cantwell's logic \cctt\ and Cooper's logic \qcctt, based on their common truth table for the indicative, avoid Gibbardian collapse altogether. This shows us that the apparent lesson from Gibbard's result--- that one has to give up Import-Export or endorse the 
material analysis of the conditional --- is mistaken. 

We confirmed that diagnosis by looking at these logics in the context of the strengthening of Gibbard's result proposed by \citet{Fitelson2013}. Specifically, we have re-interpreted Fitelson's result as showing the impossibility of having two distinct connectives that both satisfy a set of characteristic properties (Conjunction Elimination, Import-Export), and where the weaker one already satisfies Conditional Elimination. 
A logic of indicative conditionals does not have to choose between forswearing Import-Export and embracing the material conditional analysis: trivalent logics of conditionals offer a simple, yet articulate and fully truth-functional alternative that avoids both problems. 
To be sure, one might still have objections to Import-Export but, whatever they are, they cannot be supported by Gibbard-style collapse arguments. 



\bibliographystyle{../BibFolder/mychicago-ff}

\renewcommand{\bibname}{References}
\bibliography{../BibFolder/v-c(14)}

\appendix

\input{Import-Export}

\input{TechnicalAppendix}
\input{khoo}

\end{document}

%% file: Intro.tex
\section{Introduction}

The Law of Import-Export 
denotes the principle that a right-nested conditional of the form $A\to (B \to C)$ is logically equivalent to the simple conditional $(A\wedge B) \to C$ where both antecedents are united by conjunction. The Law holds in classical logic for material implication, and if there is a 
logic for the indicative conditional of ordinary language, it appears Import-Export ought to be a part of it. For instance, to use an example from \citep[300]{cooper1968propositional}, the sentences ``If Smith attends and Jones attends then a quorum will be present'', and ``if Smith attends, then if Jones attends, a quorum will be present'' appear to convey the same hypothetical information. The same appears to hold more generally, at least when $A$, $B$ and $C$ themselves are non-conditional sentences, and the equivalence has been described as ``a fact of English usage'' \citep{mcgee1989conditional}.\footnote{Import-Export has been challenged on linguistic grounds, see for instance \cite{khoo2019triviality}, drawing on examples from Fitelson. The alleged counterexamples are subtle, however, and even Khoo and Mandelkern accept a version of the law. See also Appendix \ref{sec:IE}.}  


In a celebrated paper, however, Allan \citet{gibbard1980two} showed that a binary conditional connective `$\to$' collapses to the material conditional of classical logic `$\supset$' if the following conditions hold: (i) the conditional connective satisfies Import-Export, (ii) it is at least as strong as the material conditional 
($A\to C \models_{L} A\supset C$), where $\models_{L}$ is the consequence relation of the target logic of conditionals, (iii) it is supraclassical in the sense that it reproduces the valid inferences of classical logic in conditional form 
($\models_{L} A\to C$ whenever $A\models_{\sf CL} C$). From (i)--(iii) and some natural background assumptions, Gibbard infers $A\supset C \models_{L} A\to C$. Given (ii), $\to$ and $\supset$ are thus logically equivalent, 
according to the logic of conditionals ($\models_{L}$) under consideration. Prima facie, the conditional then needs to support all inference schemes validated by the material conditional in classical logic. However, inferences such as $\neg A \models A \to C$ (one of the paradoxes of material implication) enjoy little plausibility in ordinary reasoning with conditionals. 

Gibbard's result poses a challenge for theories that compete with material implication as an adequate analysis of the indicative conditional.\footnote{Notable defenders of the material implication analysis are \citet{Lewis1976}, \citet{Jackson1979} and \citet{Grice1989}.} For example, Stalnaker's logic {\sf C2} \citep{stalnaker1968} and Lewis's logic {\sf VC} \citep{lewis1973} are both supraclassical and make the conditional stronger than the material conditional, but they invalidate Import-Export for that matter. 

Not all theories make that choice, however. All of the above logics operate in a \textit{bivalent} logical setting, 
thus limiting their options. In this paper, we explore how certain \textit{trivalent} logics of conditionals address Gibbard's challenge. These logics, which retain truth-functionality, analyze an indicative conditional of the form ``if $A$ then $C$'' as a \textit{conditional assertion} that is \emph{void} if the antecedent turns out to be false, and that takes the truth value of the consequent $C$ if $A$ is true \citep{reichenbach1935wahr,definetti1936logique,quine1950methods,belnap1970conditional}. This analysis assigns a third truth value (``neither true nor false'') to such ``void'' assertions, and gives rise to various logics that combine a truth-functional conditional connective with existing frameworks for trivalent logics \citep[e.g.,][]{cooper1968propositional,farrell1979material, milne1997bruno,cantwell2008logic,baratgin2013uncertainty,EgreRossiSprenger2020a,EgreRossiSprenger2020b}. 

This chapter clarifies the scope and significance of Gibbardian collapse results with specific attention to such trivalent logics, in which the conditional is undefined when its antecedent is false. We begin with a precise explication of Gibbard's result, including a more formal version of his original proof sketch (Section \ref{sec:Gibbard}). Then we present two trivalent logics of indicative conditionals, paired with Strong Kleene semantics for conjunction and negation, and we examine how they deal with  Gibbardian collapse (Section \ref{sec:Tri} and \ref{sec:TriSK}). We then turn to trivalent logics that replace Strong Kleene operators with \citeauthor{cooper1968propositional}'s quasi-connectives where the conjunction of the True and the third truth value is the True (Section \ref{sec:TriQ}). Specifically, we show why rejecting superclassicality---and retaining both Import-Export and a stronger-than-material conditional---is a viable way of avoiding Gibbardian collapse. 

In the second part of the paper, we consider a recent strengthening of Gibbard's result 
due to Branden Fitelson and apply it to the above trivalent logics (Section \ref{sec:FitRes}, \ref{sec:FitInt} and \ref{sec:Block}). From this analysis it emerges that Gibbard's result may be better described as a \textit{uniqueness result}: we cannot have two conditional connectives that satisfy Import-Export as well as Conjunction Elimination, where one is strictly stronger than the other, and where the weaker (already) satisfies Modus Ponens. 
We also provide three appendices: 
Appendix \ref{sec:IE}  rebuts a recent attempt at a \textit{reductio} of Import-Export, Appendix \ref{sec:technical} provides the proofs of various lemmata stated in the paper, and Appendix \ref{sec:khoo} gives a more constrained derivation of Gibbardian collapse than his original proof, of particular relevance for the first trivalent system we discuss. For more in-depth treatment of trivalent logics of conditionals, we refer the reader to our comprehensive survey and analysis in \citet{EgreRossiSprenger2020a,EgreRossiSprenger2020b}.

%% file: Gibbard5.tex
\section{Gibbard's Collapse Result}\label{sec:Gibbard}





The Law of Import-Export is an important bridge between different types of conditionals: it permits to transform right-nested conditionals into simple ones. Import-Export is of specific interest in  suppositional accounts of indicative conditionals that assess the assertability of a conditional by the corresponding conditional probability \citep[as per Adams' thesis, viz.][]{Adams1965}. Import-Export is then an indispensable tool for providing a probabilistic analysis of embedded conditionals. 
However, when Adams' Thesis, originally limited to conditionals with Boolean antecedent and consequent, is extended to nested conditionals, Import-Export creates unexpected problems.\footnote{The unrestricted version of Adams's equation is often called Stalnaker's Thesis \citep[going back to][]{Stalnaker1970} or simply ``The Equation'', with the latter name being prevalent in the psychological literature. Adams defends it in his \citeyear{Adams1975} monograph, too.} For example, a famous result by David \citet{lewis1976probabilities} shows that combining this latter equation with the usual laws of probability and an unrestricted application of Import-Export trivializes the probability of the indicative conditional.\footnote{On the reasons to defend Import-Export in relation to probabilities of conditionals, see \citet{mcgee1989conditional} and \cite{arlo2001bayesian}. A discussion of the links between Gibbardian collapse and Lewisian triviality lies beyond the scope of this paper, but we refer to \cite{lassiter2019trivalent} for a survey of Lewisian triviality results and their treatment in a trivalent framework.} Gibbard establishes a second difficulty with Import-Export, namely that any conditional satisfying Import-Export in combination with other intuitive principles collapses to the material conditional.

Gibbard's original proof \citep[][234--235]{gibbard1980two}---in reality more of an outline---was based on semantic considerations and left various assumptions implicit. Here we provide a more formal derivation. In particular, Gibbard only stressed conditions (i)--(iii) below, but implicitly assumed two further conditions, here highlighted as (iv) and (v), as well as structural constraints on the underlying consequence relation. In what follows we use $\models_{\sf CL}$ for classical consequence, and $\equiv_L$ for the conjunction of $\models_L$ and its converse. Under (v) we mean that $\supset$ obeys a classical law whenever it obeys a classical inference or a classical metainference.\footnote{An inference is a relation between (sets of) formulae: for instance the relation between $(A\supset B)\wedge A$ and $A\wedge B$ ; a metainference is a relation between inferences, for example the relation between $A\models B$ and $\models A\supset B$.}



\begin{thm}[Gibbard] Suppose $L$ is a logic whose consequence relation $\models_{L}$ is at least transitive, with $\supset$ and $\rightarrow$ two binary operators, obeying principles (i)-(v) for every formulae $A, B, C$. Then $\rightarrow$ and $\supset$ are provably equivalent in $L$.\\ 

\begin{tabular}{lll}
(i) & $A \rightarrow (B \rightarrow C)\equiv_{L} (A \wedge B) \rightarrow C$ & Import-Export \\
(ii) & $A \rightarrow B\models_{L} (A \supset B)$ & Stronger-than-Material\\
(iii) & If $A\models_{\sf CL} B$, then $\models_{L} A \rightarrow B$ & Supraclassicality\\
(iv) & If $A\equiv_{L} A'$ then $A\rightarrow B\equiv_{L} A'\rightarrow B$ & Left Logical Equivalence \\
(v) & $\supset$ obeys classical laws in $L$ & Classicality of $\supset$ 
\end{tabular}\label{T:Gibbard}
\end{thm}

\begin{proof}\

\begin{tabular}{lll}

1. & $(A\supset B) \rightarrow (A \rightarrow
B)\equiv_{L} ((A\supset B) \wedge A)\rightarrow B$ & by (i)\\

2. &  $((A\supset B) \wedge A)\equiv_{L} (A \wedge B)$ & by (v) (classical inferences)\\

3. & $(A\supset B) \wedge A)\rightarrow B\equiv_{L} (A\wedge B)
\rightarrow B$ & by 2 and (iv)\\

4. & $ (A\wedge B) \rightarrow B \equiv_{L} (A\supset B) \rightarrow (A \rightarrow
B)$ & 1, 3 and the transitivity of $\models_L$ \\

5 & $A \wedge B\models_{\sf CL} B$ &  Conjunction Elimination \\ 

6. & $\models_L (A\wedge B) \rightarrow B$ &  5 and (iii) \\

7. & $\models_L (A\supset B) \rightarrow (A \rightarrow B)$ & 4, 6, and the transitivity of $\models_L$ \\

8. & $(A\supset B) \rightarrow (A \rightarrow B)\models_L
(A\supset B) \supset (A \rightarrow B)$ & by (ii)  \\

9. & $\models_L (A\supset B) \supset (A \rightarrow B)$ & 7, 8 and the transitivity of $\models_L$\\

%
%
%

10. & $A\supset B \models_L  A \rightarrow B$ & by 9 and (v) (classical metainference)

\end{tabular}
\end{proof}
This is not the only proof of Gibbard's result. In particular \citet{Fitelson2013} and \citet{khoo2019triviality} give more parsimonious derivations. But it closely matches the structure of his original argument: first Gibbard shows that $(A\supset B) \rightarrow (A \rightarrow B)$ is a theorem of $L$ (step 1--7), from that he derives $\models_L (A\supset B) \supset (A \rightarrow B)$ (step 8--9) and finally, he infers $A\supset B \models_L  A \rightarrow B$ (step 10). 

With Gibbard we can grant that the assumptions (ii) and (iii) introduced alongside Import-Export are fairly weak. Stronger-than-Material is shared by all theories that classify an indicative conditional with true antecedent and false consequent as false.\footnote{The name MP is sometimes used for this principle, see \citet{unterhuber2014completeness}, or \citet{khoo2019triviality} who call it Modus Ponens.
We find more appropriate to use `Stronger-than-Material' since Modus Ponens is strictly speaking a two-premise argument form. The two principles are not necessarily equivalent: in the system {\sf DF/TT} for instance, Stronger-than-Material holds but not Modus Ponens (in the form $A\to B, A\models B$). 
} Supraclassicality, a restricted version of the principle of Conditional Introduction, means that deductive relations are supported by the corresponding conditional. Even that could be weakened by just assuming the conditional to support conjunction elimination as in step 6. In section \ref{sec:FitRes} we discuss more general conditions for Gibbardian collapse proposed by Branden \citet{Fitelson2013}. 

Assumptions (iv) and (v), on the other hand, are stronger than meets the eye. While the substitution rule LLE was taken for granted by Gibbard, likely on grounds of compositionality, it raises issues in relation to counterpossibles and other forms of hyperintensionality \citep[see][]{nute1980topics,fine2012counterfactuals}. However, even if one is inclined to give up principle (iv), one may not find fault with applying it in this particular case. Similarly, (v) implies that the material conditional supports classical absorption laws (step 2 of the proof) and (meta-inferential) Modus Ponens (step 10) in $L$ --- two properties not necessarily retained in non-classical logics.

Gibbard's result also leaves a number of questions unanswered. One of them concerns the implication of the mutual entailment between $\rightarrow$ and $\supset$. Does the collapse imply that the two conditionals can be replaced by one another in all contexts, for example? The answer to this question is in fact negative, as we proceed to show using trivalent logic in the next section.

%% file: TrivalentSK2.tex
\section{The Trivalent Analysis of Indicative Conditionals}\label{sec:Tri}

From his result, Gibbard drew the lesson that if we want the indicative conditional to be a propositional function, and to account for a natural reading of embedded indicative conditionals, then the function must be `$\supset$', namely the bivalent material conditional. We disagree with this conclusion: trivalent truth-functional accounts of the conditional can satisfy Import-Export and yield a reasonable account of embeddings without collapsing to the material conditional. We now explain why one may want to adopt such an approach, and then, in the next two sections, how they deal with Gibbard's result. 

\begin{table}[htb]
\[
\begin{tabular}[t]{l|ccc}
$f_{\to_{\sf DF}}$ & 1 & $\nicefrac{1}{2}$ & 0\\
 \hline
 1 & 1 &  $\nicefrac{1}{2}$ & 0\\
$\nicefrac{1}{2}$ & $\nicefrac{1}{2}$ &  $\nicefrac{1}{2}$& $\nicefrac{1}{2}$\\
 0 & $\nicefrac{1}{2}$ &  $\nicefrac{1}{2}$& $\nicefrac{1}{2}$\\
 \end{tabular}
 \qquad
 \begin{tabular}[t]{l|ccc}
$f_{\to_{\sf CC}}$ & 1 & $\nicefrac{1}{2}$ & 0\\
 \hline
 1 & 1 &  $\nicefrac{1}{2}$ & 0\\
$\nicefrac{1}{2}$ & 1& $\nicefrac{1}{2}$& 0\\
 0 & $\nicefrac{1}{2}$ &  $\nicefrac{1}{2}$& $\nicefrac{1}{2}$\\
 \end{tabular}
 \]\caption{\footnotesize Truth tables for the de Finetti conditional (left) and the Cooper-Cantwell conditional (right).}\label{tab:fincoo}
 \end{table}


Reichenbach and de Finetti proposed to analyze an indicative conditional ``if $A$, then $C$'' as an assertion about $C$ upon the supposition that $A$ is true. Thus the conditional is true whenever $A$ and $C$ are true, and false whenever $A$ is true and $C$ is false. When the supposition (=the antecedent $A$) turns out to be false, there is no factual basis for evaluating the conditional statement, and therefore it is classified as neither true nor false. This basic idea gives rise to various truth tables for $A \to C$. Two of them are the table proposed by Bruno \citet{definetti1936logique} and the one proposed independently by William \citet{cooper1968propositional} and John \citet{cantwell2008logic} (see Table \ref{tab:fincoo}). In both of them the value $\nicefrac{1}{2}$ can be interpreted as ``neither true nor false'', ``void'', or ``indeterminate''. There is moreover a systematic correspondence and duality between those tables: whereas de Finetti treats ``not true'' antecedents ($<1$) in the same way as false antecedents ($=0$), Cooper and Cantwell treat ``not false'' antecedents ($>0$) in the same way as true ones ($=1$). Thus in de Finetti's table the second row copies the third, whereas in Cooper and Cantwell's table it copies the first.


\begin{table}[ht]
\[
\begin{tabular}{c|c}
& $f_{\neg}$\\
\hline
$1$ & $0$\\
$\nicefrac{1}{2}$ & $\nicefrac{1}{2}$\\
$0$ & $1$\\
\end{tabular}
\hspace{14pt}
\begin{tabular}{c|ccc}
$f_{\wedge}$ & $1$ & $\nicefrac{1}{2}$ & $0$\\
\hline
$1$ & $1$ & $\nicefrac{1}{2}$ & $0$\\
$\nicefrac{1}{2}$ & $\nicefrac{1}{2}$ & $\nicefrac{1}{2}$ & $0$\\
$0$ & $0$ & $0$ & $0$\\
\end{tabular}
\hspace{14pt}
\begin{tabular}{c|ccc}
$f_{\supset}$ & 1 & $\nicefrac{1}{2}$ & 0\\
 \hline
 1 & 1 &  $\nicefrac{1}{2}$ & 0\\
$\nicefrac{1}{2}$ & 1 &  $\nicefrac{1}{2}$& $\nicefrac{1}{2}$\\
 0 & 1 & 1 & 1\\
 \end{tabular}
\]
\caption{\footnotesize Strong Kleene truth tables for negation, conjunction, and the material conditional.}\label{tab:SK}
\end{table}

One way to define the other logical connectives is via the familiar Strong Kleene truth tables (see Table \ref{tab:SK}). Conjunction corresponds to the ``minimum'' of the two values, disjunction to the ``maximum'', and negation to inversion of the semantic value. In particular, beside the indicative conditional $A \to C$, the trivalent analysis also admits a Strong Kleene ``material'' conditional $A \supset C$, definable as $\neg (A \wedge \neg C)$ (see again Table \ref{tab:SK}). 

To make a logic, however, we also need a definition of validity. This question is non-trivial in a trivalent setting since preservation of (strict) truth is not the same as preservation of non-falsity. Like Cooper and Cantwell, and based on independent arguments,\footnote{All other consequence relations come with problematic features \citep[Fact 3.4 in][]{EgreRossiSprenger2020a}: they either fail the Law of Identity (i.e., $\not{\models} A \to A$), or they license the inference from a conditional to its converse (i.e., $A \to C \models C \to A$).} we opt for a tolerant-to-tolerant ({\sf TT-}) consequence relation where non-falsify is preserved: an inference $A \models C$ is valid if, for any evaluation function (of the appropriate kind) $v$ from the sentences of the language to the values $\{0, \nicefrac{1}{2}, 1\}$, whenever $v(A) \in \{\nicefrac{1}{2}, 1\}$, then also  $v(C) \in \{\nicefrac{1}{2}, 1\}$. This choice yields two logics depending on how the conditional is interpreted: the logic {\sf DF/TT} based on de Finetti's truth table, and the logic {\sf CC/TT} based on the Cooper-Cantwell table.\footnote{The system {\sf CC/TT} actually matches Cantwell's system. Cooper's logic, called {\sf OL} rests on a different choice of truth tables for conjunction and disjunction, and restricts valuations to two-valued atoms.} 

Both logics make different predictions, but they agree on a common core, and they give a smooth treatment of nested conditionals. In particular both {\sf DF/TT} and  {\sf CC/TT} satisfy the Law of Import-Export. We now investigate how they deal with Gibbardian collapse.  






\section{Gibbardian collapse in {\sf DF/TT} and {\sf CC/TT}
}\label{sec:TriSK}

\noindent We first consider Gibbard's triviality result in the context of {\sf DF/TT} with its indicative and material conditionals. {\sf DF/TT} is contractive, reflexive, monotonic and transitive. An inspection of the principles (i)--(v) in Theorem \ref{T:Gibbard} shows that:
\begin{itemize}
\itemsep=0pt
\item Assumption (i) holds. In particular, both sides of the Law of Import-Export receive the same truth value in any {\sf DF}-evaluation. 
\item Assumption (ii) also holds: if there is a {\sf DF}-evaluation $v$ such that $v(A \supset B) = 0$, then $v(A) = 1$ and $v(B) = 0$, but then $v(A \rightarrow B) = 0$ as well, thus failing to make $A \rightarrow B$ tolerantly true.

\item Assumption (iii) holds in {\sf DF/TT}. 
We prove this in Appendix \ref{sec:technical}.
\item Assumption (iv) \textit{fails} in {\sf DF/TT}. In fact, $A \models_{\sf DF/TT} B$ and $B \models_{\sf DF/TT} A$ if, for any {\sf DF}-evaluation $v$, one of the following is given: 
\begin{align*}
\text{(a) } \; & v(A) = 1 = v(B) & \text{(c) } \; & v(A) = 1; \; v(B) = \nicefrac{1}{2} \\
\text{(b) } \; & v(A) = \nicefrac{1}{2} = v(B) & \text{(d) } \; & v(A) = \nicefrac{1}{2}; \; v(B) = 1 
\end{align*}
Therefore, letting $v(C) = 0$, cases (c) and (d) provide counterexamples since either $A \rightarrow C \not\models_{\sf DF/TT} B \rightarrow C$ or $B \rightarrow C \not\models_{\sf DF/TT} A \rightarrow C$.
A concrete example is the following: 
\begin{align*}
p \vee \neg p &\, \models_{\sf DF/TT} (p \rightarrow \neg p) \vee (\neg p \rightarrow p)\\
(p \rightarrow \neg p) \vee (\neg p \rightarrow p) &\, \models_{\sf DF/TT} p \vee \neg p
\end{align*}
but 
\[[(p \rightarrow \neg p) \vee (\neg p \rightarrow p)] \rightarrow (p \wedge \neg p) \not\models_{\sf DF/TT} (p \vee \neg p) \rightarrow (p \wedge \neg p)\]
\item Assumption (v) \textit{fails} in general of $\supset$ in {\sf DF/TT}. In particular, step 2 of Gibbard's proof fails: $(A \supset B) \wedge A \not\models_{\sf DF/TT} A \wedge B$, assuming $v(A)=\half$ and $v(B)=0$.
\end{itemize}
The failure of Gibbard's conditions (iv) and (v) may seem to make {\sf DF/TT} irrelevant for the discussion of his result. But this is not so: despite assumptions (iv) and (v) failing for {\sf DF/TT}'s indicative conditional and material conditional, the two conditionals turn out to be equivalent. More precisely, {\sf DF/TT} validates the equivalence of $A \supset B$ and $A \rightarrow B$, as a reciprocal entailment ($\equiv_{\sf DT/TT}$), as a material biconditional (denoted by $\bim$), and as an indicative biconditional (denoted by $\leftrightarrow$).
\begin{lem}
For every $A, B \in {\sf For}(L)$: 
\begin{align*}
 A \supset B &\equiv_{\sf DF/TT}  A \rightarrow  B\\
&\models_{\sf DF/TT} (A \supset B)  \bim  (A \rightarrow B)\\
&\models_{\sf DF/TT} (A \supset B)  \leftrightarrow  (A \rightarrow B)
\end{align*}
\label{L:DF/TTcollapse}
\end{lem}
\noindent This result in not a coincidence. As it turns out, Gibbard's result can be derived only using principles (i), (ii), (iii), (v) and structural assumptions on logical consequence, in such a way that all uses of (v) are {\sf DF/TT} sound. This result directly follows from the version of Gibbard's result established by \citet{khoo2019triviality}, as we prove in Appendix \ref{sec:khoo}. We also give a sequent-style proof of the collapse in Appendix \ref{sec:technical}, making use of the system presented in our \citet{EgreRossiSprenger2020b}.

However, such an extended form of equivalence between the indicative and the material conditional in {\sf DF/TT} does not mean that the two conditionals are identified with each other or indistinguishable. In fact, they obey very different logical principles, such as the following connexive law: 
\[A \to B \models_{\sf DF/TT} \neg (A \to \neg B) \hspace{.5cm} \mbox{ but } \hspace{.5cm} \neg A \vee B \not\models_{\sf DF/TT} \neg (\neg A \vee \neg B).\]
This shows that indicative and  material conditional cannot be validly replaced in complex formulae in {\sf DF/TT}. Put differently, {\sf DF/TT} fails the classical 
principle of replacement of equivalents. 

What is, then, the import of {\sf DF/TT}'s equivalences between different conditionals? Not much, one might argue. A look at the {\sf DF} semantics and the status of the premises of Gibbard's Theorem in {\sf DF/TT} shows that such equivalences are largely a byproduct of (i) the fact that the {\sf DF} truth table assigns value $0$ to an indicative conditional in the same cases in which it assigns value $0$ to a material conditional, and (ii) 
the fact that the tolerant-tolerant consequence relation does not distinguish between value $1$ and $\nicefrac{1}{2}$. 

Notably, things are different when we move to {\sf CC/TT}, keeping the tolerant-tolerant notion of consequence fixed, but moving to a truth-table for the conditional which assigns value $0$ to the indicative conditional in more cases. Like {\sf DF/TT}, {\sf CC/TT} is contractive, reflexive, monotonic and transitive. Moreover:
\begin{itemize}
\itemsep=0pt
\item Assumption (i) and (ii) hold in {\sf CC/TT} for the same reasons 
as {\sf DF/TT}.
\item Assumption (iii) \textit{fails} in {\sf CC/TT}. For example, $A \wedge \neg A \models_{\sf CL} B$, but $\not\models_{\sf CC/TT} (A \wedge \neg A) \rightarrow B$. A {\sf CC}-evaluation $v$ s.t. $v(A) = \nicefrac{1}{2}$ and $v(B) = 0$ provides a counterexample.
\item Assumption (iv) holds in {\sf CC/TT}. As in the {\sf DF/TT} case, we have that $A \models_{\sf CC/TT} B$ and $B \models_{\sf CC/TT} A$ if, for any {\sf CC}-evaluation $v$, one of the following is given: 
\begin{align*}
\text{(a) } \; & v(A) = 1 = v(B) & \text{(c) } \; & v(A) = 1; \; v(B) = \nicefrac{1}{2} \\
\text{(b) } \; & v(A) = \nicefrac{1}{2} = v(B) & \text{(d) } \; & v(A) = \nicefrac{1}{2}; \; v(B) = 1 
\end{align*}
However, the row of value $1$ is identical to the row of value $\nicefrac{1}{2}$ in {\sf CC}-truth tables of the indicative conditional. Therefore, whenever one of (a)--(d) holds, for every formula $C$, we have that $v(A \rightarrow C) = v(B \rightarrow C)$, proving the claim. 
\item Assumption (v) \textit{fails} in {\sf CC/TT}, for the same reason it fails in {\sf DF/TT}.
 \end{itemize}
One of (i)--(iv) thus fails for {\sf CC/TT} as it does for {\sf DF/TT}, and (v) fails in both. The failure of assumption (iii), supraclassicality, is irrelevant for blocking the proof since the only classically valid inference required for the proof is Conjunction Elimination 
($A \wedge B \models A$). This inference is also validated by {\sf CC/TT}. The proof is thus blocked exclusively by the failure of assumption (v): $\supset$ does not behave classically in {\sf CC/TT} (i.e., step 2 in our reconstruction of Gibbard's proof fails). 
Unlike \dftt, \cctt avoids Gibbardian collapse: it declares both conditionals  materially equivalent, but neither logically equivalent nor equivalent according to the indicative biconditional:

\begin{lem}
For every $A, B \in {\sf For}(L)$:
\begin{align*}
A \to B & \models_{\sf CC/TT} A \supset B \hspace{5pt} \text{ but } \hspace{5pt} A \supset B \not\models_{\sf CC/TT} A \rightarrow B \\
& \models_{\sf CC/TT} (A \supset B) \bim (A \rightarrow B) \\
&\not\models_{\sf CC/TT} (A \supset B) \leftrightarrow (A \rightarrow B)
\end{align*}
\label{L:CC/TTcollapse}
\end{lem} 

In general, the indicative conditional of {\sf CC/TT} is \emph{strictly stronger} than its material counterpart: $A \rightarrow B$ entails $A \supset B$, but is not entailed by it. And this is, by the light of a logic of  indicatives, a welcome result: the paradoxes of material implication consist, for the most part, of conditional statements that are clearly unacceptable, but are declared valid by the material conditional analysis. The Cooper-Cantwell analysis validates \emph{fewer} conditional principles (`fewer' in the sense of inclusion), and avoids the most problematic paradoxes. 

Altogether, {\sf DF/TT} and {\sf CC/TT} avert Gibbardian triviality in different ways. In both of them the material conditional is not fully classical, but an extensional collapse takes place in 
{\sf DF/TT} anyway; this, however, does not make the material conditional always replaceable by the indicative in {\sf DF/TT}. On the other hand, the indicative conditional of {\sf CC/TT} is more remote from its material counterpart: not only does it validate different conditional principles (removing the most pressing paradox of material implication), it is also  extensionally distinct from the material conditional within {\sf CC/TT} itself. 

Summing up, while Gibbardian collapse is avoided more markedly in {\sf CC/TT} than in {\sf DF/TT}, in neither logic does it constitute a form of ``triviality'': even when indicative and material conditionals are declared to be equivalent, they are firmly set apart by  their inferential behavior. 
This concludes our study of Gibbard's original collapse result in trivalent logics based on Strong Kleene connectives. In the next section, we expand the scope of our analysis and look at trivalent logics of conditionals with a different semantics for the standard logical connectives.

%% file: TrivalentQ.tex
\section{Gibbardian Collapse in {\sf QCC/TT}
}\label{sec:TriQ}

The logics {\sf DF/TT} and {\sf CC/TT} solve a large set of problems related to the indicative conditional, but they also have important limitations.  First, both \cctt\ and \dftt\ validate the  Linearity principle $(A\to B) \vee (B\to A)$ for arbitrary $A$ and $B$. This schema was famously criticized by \citet{maccoll1908if}: neither of ``if John is red-haired, then John is a doctor'' and ``if John is a doctor, then he is red-haired'' seems acceptable in ordinary reasoning. So it is unclear on which basis we should accept, or declare as true, the disjunction of both sentences. Imagine, for example, that John is a black-haired doctor or a red-haired carpenter.  

In a similar vein, some highly plausible conjunctive sentences can never be true on \dftt\ or \cctt. The schema $(A \to A) \wedge (\neg A \to \neg A)$ (``if A, then A; and if $\neg$A, then $\neg$A'') is always classified as neither true nor false, although each of the conjuncts is a \dftt{}- and \cctt{}-theorem.\footnote{We are indebted to Paolo Santorio for this example.} Likewise, an ensemble of conditional predictions of the form $(A \to B) \wedge (\neg A \to C)$ will always be indeterminate or false \citep[][368--370]{Bradley2002}. However, a sentence such as: 
\begin{exe}
\ex If the sun shines tomorrow, Paul will go to the office by bike; and if it rains, he will take the metro. \label{ex:beach}
\end{exe}
seems to be true (with hindsight) if the sun shines tomorrow and Paul goes to the office by bike.

\begin{table}[ht]
\[
\begin{tabular}{c|ccc}
$f'_{\wedge}$ & $1$ & $\nicefrac{1}{2}$ & $0$\\
\hline
$1$ & $1$ & $1$ & $0$\\
$\nicefrac{1}{2}$ & $1$ & $\nicefrac{1}{2}$ & $0$\\
$0$ & $0$ & $0$ & $0$\\
\end{tabular}
\hspace{10pt}
\begin{tabular}{c|ccc}
$f'_{\vee}$ & $1$ & $\nicefrac{1}{2}$ & $0$\\
\hline
$1$ & $1$ & $1$ & $1$\\
$\nicefrac{1}{2}$ & $1$ & $\nicefrac{1}{2}$ & $0$\\
$0$ & $1$ & $0$ & $0$\\
\end{tabular}
\hspace{10pt}
\begin{tabular}{c|ccc}
$f'_{\supset}$ & $1$ & $\nicefrac{1}{2}$ & $0$\\
\hline
$1$ & $1$ & $0$ & $0$\\
$\nicefrac{1}{2}$ & $1$ & $\nicefrac{1}{2}$ & $0$\\
$0$ & $1$ & $1$ & $1$\\
\end{tabular}
\]
\caption{\footnotesize Truth tables for trivalent quasi-conjunction and quasi-disjunction and the material conditional based on quasi-disjunction, as advocated by \citet{cooper1968propositional}.}\label{tab:quasi}
\end{table}

A principled reply to these challenges consists in modifying the truth tables for trivalent conjunction and disjunction, as proposed by \citet{cooper1968propositional} 
(see also \citealt{dubois1994conditional} and \citealt{calabrese2002deduction}). In these truth tables, reproduced in  Table \ref{tab:quasi}, the conjunction of 
value $1$ and 
value $\nicefrac{1}{2}$
is 
value $1$, and vice versa for disjunction. This is coherent with the idea that a conditional assertion with two components (e.g., in Bradley's examples) should be classified as true if one of the assertions came out true, and the other one void. Notably, the material conditional $A \supset C$ (definable as $\neg A\vee B$ or as $\neg (A\wedge \neg B)$) of a {\sf TT-}logic based on these quasi-connectives blocks the paradoxes of material implication ($\neg A \not{\models} A \supset C$, $C \not{\models} A \supset C$), in line with the failure to validate Disjunction Introduction. 

 Adopting ``quasi-conjunction'' and ``quasi-disjunction'' \citep[the terminology is due to][]{Adams1975} invalidates Linearity and gives non-trivial truth conditions for ensembles or partitions of conditional assertions. In particular, $(A \to A) \wedge (\neg A \to \neg A)$ is always true, and so is $(A \to B) \wedge (\neg A \to C)$ when one of its conjuncts is true. We call the resulting logics  {\sf QDF/TT} and {\sf QCC/TT}.\footnote{{\sf QCC/TT} is almost identical to Cooper's logic {\sf OL}, except that Cooper requires valuations to be bivalent on atomic formulae.}  However, when paired with \dftt, quasi-conjunction leads to a violation of Import-Export, but not so in \cctt. So the system of interest for us in this section is {\sf QCC/TT}.


How does  {\sf QCC/TT} then fare with respect to the five premises of Gibbard's proof?
\begin{itemize}
\itemsep=0pt
\item Assumption (i) holds since both sides of the Law of Import-Export receive the same truth value in any {\sf QCC}-evaluation. 
\item Assumption (ii) \textit{fails} since the (quasi-)material conditional is strictly stronger than the indicative conditional. The valuation $v(A) = 1$ and $v(B) = \nicefrac{1}{2}$ is a model of $A \to B$, but not of $A \supset B$, which takes the same truth values as $\neg A \vee B$.
\item Assumption (iii) and (v) \textit{fail} with the same countermodels as in {\sf CC/TT}. 
\item Assumption (iv) holds: it is independent of the interpretation of the standard connectives and the proof for {\sf CC/TT} can be transferred. 
\end{itemize}
In {\sf QCC/TT}, two steps of Gibbard's proof are blocked, corresponding to the failure of  assumptions (ii) and (v). Like before, the failure of (iii) is inessential since the proof just requires Conjunction Elimination instead of the more general property of Supraclassicality.

\begin{table}[ht]
\[
\begin{tabular}{c|ccccccc|ccc}
Condition & (i) & (ii) & (iii) & CE & (iv) & (v) & TRM & $\equiv$? & $\leftrightarrow$? & $\bim$?  \\ \hline
\dftt\ & \cmark & \cmark & \cmark & \cmark & \xmark & \xmark & \cmark & \cmark & \cmark & \cmark\\
\cctt\ & \cmark & \cmark & \xmark & \cmark & \cmark & \xmark & \cmark & \xmark & \xmark & \cmark\\
\qcctt\ & \cmark & \xmark & \xmark & \cmark & \cmark & \xmark & \cmark & \xmark & \xmark & \xmark\\
\end{tabular}\]
\captionsetup{singlelinecheck=off}
\caption[]{\footnotesize Overview of which premises of Gibbard's proof are satisfied by the logics \dftt, \cctt and \qcctt. CE = conjunction elimination (=a sufficient surrogate for (iii)), TRM = transitivity, monotonocity and reflexivity of the logic. $\equiv, \leftrightarrow, \bim$ concern whether logical, indicative, or material equivalence holds between $\supset$ and $\to$.
} 
\end{table}

Since the material conditional is strictly stronger than the indicative in \qcctt, Gibbardian collapse does not happen, and moreover, neither the material nor the indicative conditional declares the two connectives equivalent:
\begin{lem}
For every $A, B \in {\sf For}(L)$: 
\begin{align*}
A \supset B & \models_{\sf QCC/TT} A \rightarrow B \hspace{5pt} \text{ but } \hspace{5pt} A \to B \not{\models}_{\sf QCC/TT} A \supset B \\
& \not{\models}_{\sf QCC/TT} (A \supset B) \bim (A \rightarrow B) & & \\
& \not{\models}_{\sf QCC/TT} (A \supset B) \leftrightarrow (A \rightarrow B) & & 
\end{align*}
\end{lem}
In \qcctt, the connectives are thus more distinct than in \dftt\ (where they are logically and materially equivalent) and \cctt\ (where they are not logically, but still materially equivalent). The way out provided by \qcctt\ is notable for another reason, too. Most theorists react to Gibbardian collapse either by giving up or restricting Import-Export (e.g., Stalnaker, Kratzer), or by endorsing a material implication analysis of the indicative conditional (e.g., Grice, Lewis, Jackson). Denying that $\supset$ satisfies the classical laws in a logic of conditionals---the road taken by \cctt---is already less common. However, 
Cooper's 
original approach is probably unique in entertaining the possibility of an indicative conditional that is strictly weaker than the material conditional. The explanation is probably that bivalent logic has been the default framework for formal work on conditionals and the material conditional represents, in that framework, the weakest possible conditional connective. 
The logic \qcctt\ thus shows an original and surprising way of defining the relationship between the two connectives.


%% file: Fitelson.tex
\section{Fitelson's Generalized Collapse Result}\label{sec:FitRes}

\noindent Our rendition of Gibbard's original argument has revealed that one of the premises ---namely that $\models A \to C$ whenever $A$ classically implies $C$--- is stronger than needed: we only require that $(A \wedge C) \to C$ be a logical truth. On the other hand, Gibbard's argument uses some properties of classical logic and the material conditional, such as the fact that $A \wedge (A \supset C)$ is logically equivalent to $A \wedge C$. Gibbard's result can thus be generalized along two dimensions: first, use premises only as strong as we need them for the proof of the collapse result; second, make explicit the classicality assumptions (compare Section \ref{sec:Gibbard}) and extend the result to other logics than just classical logic with the material conditional. 

Branden \citet{Fitelson2013} has provided one such generalized result. 
It concerns the relation between two binary connectives represented by the symbols $\scc $ and $\wcc$ in an arbitrary logic $L$, whose consequence relation we denote with
$\models_L$.
Letting $A$, $B$ and $C$ stand for arbitrary formulae of $L$, 
and $\models_L$ for some consequence relation defined for the language of $L$, Fitelson states  eight conditions sufficient to derive a general collapse result:\footnote{Our notation swaps the meaning of the symbols $\scc$ and $\wcc$ in Fitelson's work to make it consistent with the rest of our paper.}
\begin{enumerate}
  \item[(1)] $\models_L (A \wedge B) \wcc  A$ \hfill (Conjunction Elimination for $\wcc $)
  \item[(2)] $\models_L (A \wedge B) \scc  A$ \hfill (Conjunction Elimination for $\scc $)
  \item[(3)] $\models_L A \wcc  (B \wcc  C)$ if and only if $\models_L (A \wedge B) \wcc  C$ \hfill (Import-Export for $\wcc $)
  \item[(4)] $\models_L A \scc  (B \scc  C)$ if and only if $\models_L (A \wedge B) \scc  C$ \hfill (Import-Export for $\scc $)
  \item[(5)] If $\models_L A \scc  B$, then $\models_L A \wcc  B$ \hfill ($\scc $ implies $\wcc $)
  \item[(6)] If $\models_L A \wcc  B$, then $A \models_L B$ \hfill (Conditional Elimination 
  for $\wcc $). 
  \item[(7)] If $A \equiv_L B$ and $\models_L A \scc  C$, then also $\models_L B \scc   C$\hfill  (Left Logical Equivalence) 
  \item[(8)] If $A \models_L B$ and $A \models_L C$, then $A \models_L B \wedge C$ \hfill (Conjunction Introduction)
\end{enumerate}
In short, Fitelson's result concerns the relationship between two conditionals which satisfy both Conjunction Elimination (1+2) and Import-Export (3+4), and of which one is stronger than the other one (5). The stronger conditional, represented by the normal arrow $\to$, is supposed to represent the indicative conditional. 
Moreover, it is assumed that the weaker connective $\wcc $ satisfies Conditional Elimination relative to the 
logic $\models_L$ (6), and that one can substitute 
$\models_L$-equivalents in the premises of 
$\scc $-validities (7).\footnote{What we call Conditional Elimination is the converse of Conditional Introduction. The two properties together are known as the Deduction Theorem. Conditional Elimination corresponds to (meta-inferential) Modus Ponens.
} Finally, it assumes Conjunction Introduction (8), a very natural property: if two propositions follow from a third, then so does their conjunction.


Fitelson shows that these axioms are logically independent from each other and that they are sufficient to show that the two connectives $\scc $ and $\wcc $ are logically equivalent: 
\begin{description}
  \item[Theorem (Fitelson 2013):] From conditions (1)--(8) it follows that 
  \begin{align*}
A \wcc  B &\models_L A \scc  B \qquad \text{and}  & A \scc  B &\models_L A \wcc  B
  \end{align*}
\end{description}
As Fitelson emphasizes, this should not be taken to imply that the connective $\scc $ collapses to the \textit{material} conditional, or that the indicative conditional ``If A, then C'' should be interpreted as ``not A or C''. Fitelson's result is interpretation-neutral and concerns \textit{any} two connectives with the said properties; specifically, it does not presuppose that the weaker connective $\wcc$ corresponds to the material conditional $\supset$. Whether the material conditional $A \supset C$ (i.e., $\neg A \vee C$) satisfies the properties of $\wcc $ (i.e., conditions (1), (3), (5) and (6)) will depend on 
which logic we choose to interpret $\models_L$, and we will soon see that it need not in a trivalent setting. What Fitelson shows is rather that if a conditional connective satisfies Conjunction Elimination, Import-Export and Modus Ponens, then in any logic with Conjunction Introduction, there cannot be a strictly stronger conditional connective that satisfies these conditions as well as axiom (7)---the substitution of equivalents in the premises of its theorems. In this sense, Fitelson proves the existence of an upper bound for the strength of a conditional that satisfies these intuitively desirable logical properties. Moreover Fitelson shows that such a connective must also validate some central intuitionistic principles. 

\section{Fitelson's Result and Trivalent Logic}\label{sec:FitInt}

What does Fitelson's result mean for trivalent logics when his two connectives $\scc$ and $\wcc$ are identified with the indicative and the material conditional? Keeping the tolerant-to-tolerant character of the logical consequence relation fixed (see Section \ref{sec:TriSK} for why), we have to assign values to the following parameters: 
\begin{itemize}
  \item the truth table for the indicative conditional (de Finetti or Cooper-Cantwell);
  \item the truth table for conjunction and disjunction (Strong Kleene operators or Cooper's quasi-conjunction and disjunction);
  \item which connective in Fitelson's result represents the indicative conditional, and which connective represents the material conditional.
\end{itemize}
This leaves us with eight different logics, characterized by the choice of the truth table for the indicative conditionals ({\sf DF} or {\sf CC}), the truth tables for conjunction and disjunction (Strong Kleene or Cooper), and the assignment of conditionals to Fitelson's connectives ($\scc $ and $\wcc $). Fitelson suggests that the stronger connective $\scc$ stands for the indicative conditional. However, the properties of $\wcc$, which include Modus Ponens, Conjunction Elimination and Import-Export, could also square well with the indicative conditional. Moreover, the indicative conditional can be \textit{weaker} than the material conditional in \qcctt. Thus, we have to carefully examine all ways of distributing Fitelson's connectives to truth tables. 

As noticed in the previous section, {\sf QDF/TT} does not satisfy Import-Export for the indicative conditional and so we set it aside (either condition (3) or condition (4) will fail). All the other logics satisfy conditions (1)--(4) and also condition (8). 
Thus our discussion will be limited to those logics and the more controversial properties (5), (6) and (7). Actually, we see that none of our trivalent logics satisfies all of these principles:
\begin{description}
  \item[DF/TT with $\scc =\rightarrow_{\sf DF}$] Satisfies (5)---material and indicative conditional are {\sf DF}-equivalent---, but neither (6) nor (7). For (6), consider $|A| = \nicefrac{1}{2}$, $|B| = 0$, and for (7), consider $|A| = \half$, $|B| = 1$, and $|C| = 0$. 
  \item[DF/TT with $\scc =\supset$] Satisfies (5), but neither (6) and (7). Consider the same examples as above. 
  \item[CC/TT with $\scc =\rightarrow_{\sf CC}$] Satisfies (5) and (7), but not (6). Consider again 
  $|A| = \nicefrac{1}{2}$ and $|B| = 0$.   
  \item[CC/TT with $\scc =\supset_Q$] Satisfies (6), but neither (5) nor (7). For (5), consider $|A| = \half$ and $|B| = 0$; for (7) consider $|A| = \half$, $|B| = 1$, and $|C| = 0$.  
\item[QCC/TT with $\scc =\rightarrow_{\sf CC}$] Satisfies (6) and (7), but not (5). Consider $|A| = 1$ and $|B| = \half$. 
\item[QCC/TT with $\scc =\supset_Q$] Satisfies (5) and (6), but not (7). The counterexample is $|A| = |C| = \half$ and $|B| = 1$. 
  \end{description}
  
\begin{table}[h!tb]
\centering
\renewcommand{\tabcolsep}{.25em}
\begin{tabular}[c]{p{0.3\textwidth}c||c|c||c|c||c|c}
\multicolumn{2}{c||}{Condition/Logic}   & \multicolumn{2}{c||}{\sf DF/TT} &  \multicolumn{2}{c||}{\sf CC/TT} & \multicolumn{2}{c}{\sf QCC/TT} \\ \hline \hline
Assignment of Symbols & Indicative=?   & $\scc $ & $\wcc $  & $\scc $ & $\wcc $  & $\scc $ & $\wcc $ \\ 
 & Material=?  & $\wcc $ & $\scc $ & $\wcc $ & $\scc $ & $\wcc $ & $\scc $ \\ \hline \hline
\multicolumn{2}{l||}{(5): $\scc $ implies $\wcc $} & \cmark & \cmark & \cmark & \xmark &  \xmark & \cmark \\ 
\multicolumn{2}{l||}{(6): Conditional Elimination for $\wcc $} & \xmark & \xmark & \xmark & \cmark &  \cmark & \cmark \\ 
\multicolumn{2}{l||}{(7): Substitution of Equivalents ($\scc $)}  &  \xmark & \xmark & \cmark & \xmark & \cmark & \xmark \\ \hline \hline
\multicolumn{2}{l||}{Collapse strongly blocked?} & \cmark & \cmark & \xmark & \xmark & \xmark & \xmark \\ 
\end{tabular} 
\caption{\footnotesize Overview of the satisfaction/violation of Fitelson's conditions (5)--(7) in different trivalent logics.} 
\label{tab:Fitelson} 
\end{table}  

Table \ref{tab:Fitelson} summarizes our findings. As we see, none of our trivalent candidate logics for the indicative conditional obeys all of these axioms. Since there are no obvious alternatives to the (various forms of the) material conditional as the second connective in Fitelson's theorem, Gibbardian collapse is blocked for the entire range of trivalent logics that we study. In particular, since at least one of the axioms fails for all configurations we have looked at, the connective $\scc$ must also fail one the principles of the intuitionistic conditional (this is, as mentioned above, a consequence of satisfying conditions (1)--(8)). 

\section{Blocking 
Fitelson's Collapse Strongly and Weakly}\label{sec:Block}

In order to better assess the distinct ways in which Fitelson's collapse is blocked in trivalent logics, we introduce a useful distinction. We say that a logic of indicative conditionals $L$ blocks the collapse \emph{strongly} if at least one of conditions (1)--(8) is not satisfied by letting $\scc  = \to_{\sf ind}$, where $\to_{\sf ind}$ is the connective that, in $L$, is taken to model the indicative conditional. We say that the $L$ blocks the collapse \emph{weakly} if $\scc  = \to_{\sf ind}$ \emph{and} $\wcc  = \supset$, where $\supset$ is the material conditional in $L$. In other words, $L$ blocks Fitelson's collapse strongly if some of Fitelson's premises fails in $L$ once $\scc $ is interpreted as $L$'s candidate for the indicative conditional, regardless of how the other conditional $\wcc $ is interpreted. On the other hand, $L$ blocks Fitelson's collapse only weakly if some of Fitelson's premises fails in $L$ once $\scc $ is interpreted as $L$'s candidate for the indicative conditional \emph{and} $\wcc $ is interpreted as $L$'s material conditional. In the former case, $L$'s indicative conditional is non-trivial (in the sense of the collapse) \emph{by itself}, whereas in the latter case it is non-trivial only if we assume (at least some of) the features of $\supset$ in $L$ for the other conditional. 

A glance at our findings shows that Fitelson's collapse result is blocked strongly for the {\sf DF/TT}-logics, and only weakly for all {\sf (Q)CC/TT}-logics. 
The failure of collapse in the {\sf (Q)CC/TT}-logics is due to both features of  the Cooper-Cantwell conditional in a {\sf TT}-consequence relation \emph{and} the choice of the material conditional as the interpretation of the weaker connective $\wcc$. Does this show that the indicative conditional of the {\sf (Q)CC/TT}-logics is ``trivial'', or in some sense uninteresting? Not really. All Fitelson's result can be used to argue for is that, given (1)--(8), the indicative conditional of {\sf (Q)CC/TT}-logics is $L$-equivalent to (i.e., inter-$L$-inferrable with) an unspecified conditional which: (i) cannot be the material conditional of $L$ (since {\sf (Q)CC/TT}-logics weakly block the collapse), and (ii) satisfies 
conditions (1), (3), (5), and (6), over a background logic which satisfies (8).\footnote{Respectively: Conjunction Elimination (1), Import-Export (3), being entailed by indicative conditionals (5), Modus Ponens (6), and Conjunction Introduction (8)} 
Now, not only are these properties unproblematic---by themselves, they do not give rise to any paradox of implication---, they are indeed desirable. Hence, it should actually be a welcome result that an indicative conditional is equivalent to a conditional with such properties.


In summary, since the trivalent logics we have examined block 
Fitelson's collapse result systematically, we do not find ourselves in the dilemma of having to sacrifice Import-Export, or another plausible condition to avoid triviality. To us, the most reasonable construal of Fitelson's theorem is as a \textit{uniqueness result}: it is impossible to have two conditional connectives both satisfying Import-Export and Conjunction Elimination, such that one is strictly stronger than the other and where the weaker one satisfies Conditional Elimination. This leaves  Left Logical Equivalence (condition (7)) out of the picture, but as Table \ref{tab:Fitelson} shows, this condition is only required to prevent collapse in one case, namely {\sf QCC/TT}, in which the material conditional is stronger than the indicative conditional. For all other combinations there is a tension between the relative strength of the connectives (as codified by (5)) and the fact that the weaker connective should satisfy Conditional Elimination (namely (6)).

%% file: Import-Export.tex
\section{Import-Export Revisited}\label{sec:IE}

Our case study on trivalent logics shows that it is possible to have Import-Export without restriction in a conditional logic without running into undesirable results of collapse to the material conditional, or to other connectives that are clearly too weak. Specifically, even if a conditional connective $\to$ validates Import Export, the schema $(A \wedge B) \to A$, and is stronger than the material conditional, it need not be logically equivalent to the latter. 

This observation raises the suspicion that the scope of Gibbardian collapse results may have to do with the absence and presence of bivalence. Note that the third truth value has been essential to constructing suitable counterexamples to Fitelson's conditions (1)--(8), and to blocking a generalized collapse theorem. In other words, we conjecture that Gibbardian collapse is a characteristic feature of conditional connectives with Import-Export in \textit{bivalent logic}.

This conjecture shall now be probed by studying a recent reductio argument against Import-Export. Matthew \citet{Mandelkern2020a} argues that Import-Export, when conjoined with other plausible principles, leads to absurd conclusions \citep[compare also][]{Mandelkern2019a}. Specifically, for a logic 
$(L, \models_L)$ with formulae $A$, $B$ and $C$ and a connective $\to$ representing the indicative conditional, Mandelkern considers (and defends) the following three principles: 
\begin{align}
\text{If } & & A & \models_L B & \text{then} & & & \models_L A \to B \tag{Conditional Introduction}
\\
\text{If } & & & \models_L A \to B & \text{then} & &  A \to (B \to C) & \equiv_L A \to C \tag{Nothing Added}\\
%
\text{If } & & A \to C & \equiv_L B \to C & \text{then} & & A & \equiv_L B \tag{Equivalence}
\end{align}
where $\equiv_L$ means, as before, that both $\models_L$ and its converse hold. Conditional Introduction is valid in all trivalent logics we considered, whereas Nothing Added and Equivalence hold in {\sf (Q)CC/TT}, but not in {\sf (Q)DF/TT}.\footnote{Countermodel for Nothing Added in {\sf (Q)DF/TT}: $v(A) = 1$, $v(B) = \nicefrac{1}{2}$, $v(C) = 0$. Countermodel for Equivalence in  {\sf (Q)DF/TT}: $v(A) = v(C) = 0$, $v(B) = \nicefrac{1}{2}$.}
Mandelkern requires another premise, restricted to atom-classical formulae $A$ 
(i.e. such that all propositional variables have a classical value)
, but without restrictions on $B$: 
\begin{align}
\text{For atom-classical $A$:} & & \models_L (A \wedge \neg A) \to B \tag{Quodlibet}
\end{align}
Quodlibet too holds in the trivalent logics we surveyed. From these four principles Mandelkern derives the following intermediate result: 
\begin{align}
\text{For atom-classical $A$:} & & A \models_L \neg A \to B \tag{Intermediate}
\end{align}
Intermediate also holds in {\sf CC/TT}, and plausibly so: If $A$ holds then any conditional assertion with $\neg A$ as a premise is void, and thus valid in a logic with a tolerant-to-tolerant consequence relation. Intermediate is equivalent to $\neg A\models_L A\to B$, from which Mandelkern derives:
\begin{align}
\text{For atom-classical $A$: } & & \neg (A \to B) \models_L A \tag{Ex Falso}
\end{align}

\noindent The lesson Mandelkern takes from this is:
\begin{quotation}
\noindent [Intermediate] is clearly false [...]. For this conclusion entails that the falsity of $ \neg A \to B$ entails the falsity of $A$; more succinctly (given classical negation, which is not in dispute here), the falsity of $A \to B$ entails the truth of $A$. \citep[][symbolic notation changed]{Mandelkern2020a}
\end{quotation} 

\noindent Ex Falso is definitely an unacceptable principle for a theory of indicative conditionals. As it turns out, it is invalidated in the trivalent logics, including ${\sf CC/TT}$ (Consider $v(A) = 0$.) What happened in the step from Intermediate to Ex Falso? As hinted by Mandelkern's parenthetical remark, the step is blocked in ${\sf CC/TT}$ because trivalent negation is no longer classical. In particular, {\sf TT}-consequence does not obey Contraposition. This feature suggests a tradeoff: the trivalent logics of conditionals we considered validate Import-Export without restriction, and they do not fall prey to Mandelkern's reductio. However, they no longer validate Contraposition without restriction, and because ${\sf CC/TT}$ satisfies the full Deduction Theorem, the associated conditional fails contraposition too. For indicative as well as for counterfactuals, contraposition is moot, however, in that regard the way in which Mandelkern's reductio is blocked here does not appear problematic.\footnote{Mandelkern does not dispute the validity of Import-Export for simple right-nested conditionals where it looks very compelling; he just thinks that Import-Export has less than general scope. Specifically, he has doubts about the application of Import-Export to compound conditionals with left-nesting, such as $A \to ((B \to C) \to D)$. Naturally, it is very difficult to find reliable empirical data or expert intuitions on how such sentences are, or should be, interpreted.}

%% file: TechnicalAppendix.tex
\section{Technical appendix}\label{sec:technical}
\noindent In this appendix, we first prove that assumption (iii) of Gibbard's Theorem holds in {\sf DF/TT}. Then, we give a syntactic proof of the mutual {\sf DF/TT}-entailments of $A \rightarrow B$ and $A \supset B$ (cf.~Lemma \ref{L:DF/TTcollapse}), in the three-sided sequent calculus for {\sf DF/TT} from \citep[][]{EgreRossiSprenger2020b}. The remaining claims of the Lemma are then immediate. The calculus is sound and complete for {\sf DF/TT}, so the proof immediately establishes the corresponding semantic claims, but we believe that a syntactic proof provides a good illustration of how one can, rather naturally, reason in trivalent logics. Similar proofs are available for the corresponding claims in {\sf CC/TT}.

\begin{lem} Supraclassicality holds in {\sf DF/TT}.\end{lem}

\begin{proof} 
We prove the contrapositive. Suppose $\not\models_{\sf DF/TT} A \rightarrow B$. Then there is a {\sf DF}-evaluation $v : {\sf For}(L) \longmapsto \{0, \nicefrac{1}{2}, 1\}$ s.t. $v(A) = 1$ and $v(B) = 0$. We then claim that, in this case, then there is always a \emph{classical} evaluation $v_{\sf cl} : {\sf For}(L) \longmapsto \{0, 1\}$ s.t. for every $C \in {\sf For}(L)$, if $v(C) = 1$, then $v_{\sf cl}(C) = 1$ and if $v(C) = 0$, then $v_{\sf cl}(C) = 0$, thus showing that $A \not\models_{\sf CL} B$.
We prove this by induction on the logical complexity (${\sf cp}$) of $A$ and $B$: 
\begin{itemize}
\item ${\sf cp}(A) = {\sf cp}(B) = 0$. Then, $A \rightarrow B$ has the form $p \rightarrow q$, and $v(p) = 1$, $v(q) = 0$. $v_{\sf cl}$ is any classical evaluation which agrees with $v$ on $p$ and $q$, so clearly $p \not\models_{\sf CL} q$. 
\item ${\sf cp}(A) = m+1$ and ${\sf cp}(B) = 0$. Then $A \rightarrow B$ has the form $C \rightarrow q$, for $C$ a logically complex sentence. We  assume the claim as IH up to $m$, and reason by cases: 
\begin{itemize}
\item $C$ is $\neg D$. Then $v(\neg D) = 1$ and $v(q) = 0$, and $v(D) = 0$. By IH, then, there is a classical evaluation $v_{\sf cl}$ s.t. $v_{\sf cl}(D) = 0$ and $v(q) = 0$, so that $C\not\models_{\sf CL} q$. 
\item $C$ is $D \vee E$. Then $v(D \vee E) = 1$ and $v(q) = 0$. There are several cases, all similar between them, where at least one of the disjunct receives value $1$: 
\begin{itemize}
\item $v(D) = 1$ and $v(E) = 1$
\item $v(D) = 1$ and $v(E) = \nicefrac{1}{2}$
\item $v(D) = 1$ and $v(E) = 0$
\item $v(D) = \nicefrac{1}{2}$ and $v(E) = 1$
\item $v(D) = 0$ and $v(E) = 1$
\end{itemize}
Let $X$ be the (or `a') disjunct which receives value $1$ by $v$. By IH, $v_{\sf cl}(X) = 1$, and then $v_{\sf cl}(D \vee E) = 1$ and $v_{\sf cl}(q) = 0$, hence $C \not\models_{\sf CL} q$
\item The case where $C$ has the form $D \wedge E$ is similar to the above one. 
\item $C$ is $D \rightarrow E$. Then $v(D \rightarrow E) = 1$ and $v(q) = 0$, and therefore $v(D) = v(E) = 1$. By IH, then, $v_{\sf cl} (D) = v_{\sf cl} (E) = 1$, hence $C \not\models_{\sf CL} q$. 
\end{itemize}
\item The cases where ${\sf cp}(A) = 0$ and ${\sf cp}(B) = n+1$, and where ${\sf cp}(A) = m+1$ and ${\sf cp}(B) = n+1$ are dealt with similarly. 
\end{itemize}\end{proof}

Notice that, in this proof, a {\sf DF}-evaluation for the language including the conditional is mapped to a \emph{classical} evaluation for the same language, i.e. a classical evaluation which also interpret  formulae of the form $A \rightarrow B$. However, the proof does not specify how formulae of the form $A \rightarrow B$ are classically interpreted---that is, $A \rightarrow B$ may or may not be interpreted as a classical material conditional. We also note that an attempted proof along the lines of the above one would fail for {\sf CC/TT} exactly because the conditions under which an indicative conditional receives value $0$ under a {\sf CC}-evaluation strictly exceed the conditions under which a material  conditional receives receives value $0$ under a classical evaluation, unlike in a {\sf DF}-evaluation.

\begin{lem}\label{lemma:seq}
Let $\Gamma \vdash_{\sf DF/TT} \Delta$ indicate that there is a derivation of the three-sided sequent $\Gamma \,|\, \Delta \,|\, \Delta$ in the calculus developed in \citet{EgreRossiSprenger2020b}, \S\S 3.1-3.2. Then, for every $A, B \in {\sf For}(L)$:
\begin{center}
$A \supset B \vdash_{\sf DF/TT} A \rightarrow B$ \hspace{5pt} and \hspace{5pt} $A \rightarrow B \vdash_{\sf DF/TT} A \supset B$
\end{center}
\end{lem}

\begin{proof}
We write $\neg (A \wedge \neg B)$ for $A \supset B$, as the two formulae are definitionally equivalent in {\sf DF/TT}. The following derivation establishes that $A \supset B \vdash_{\sf DF/TT} A \rightarrow B$: 

\begin{center}
\AxiomC{$\,$}
\RightLabel{\scriptsize{$\mathsf{SRef}$}}
\UnaryInfC{$A \,|\, A, B \,|\, A \rightarrow B, A$}
\RightLabel{\scriptsize{$\rightarrow$-$\nicefrac{1}{2}$}}
\UnaryInfC{$\varnothing \,|\, A \rightarrow B \,|\, A \rightarrow B, A$}
\AxiomC{$\,$}
\RightLabel{\scriptsize{$\mathsf{SRef}$}}
\UnaryInfC{$A, B \,|\, A, B \,|\, A$}
\AxiomC{$\,$}
\RightLabel{\scriptsize{$\mathsf{SRef}$}}
\UnaryInfC{$A, B \,|\, A, B \,|\, B$}
\RightLabel{\scriptsize{$\rightarrow$-1}}
\BinaryInfC{$A, B \,|\, A, B \,|\, A \rightarrow B$}
\RightLabel{\scriptsize{$\rightarrow$-$\nicefrac{1}{2}$}}
\UnaryInfC{$B \,|\, A \rightarrow B \,|\, A \rightarrow B$}
\RightLabel{\scriptsize{$\neg$-1}}
\UnaryInfC{$\varnothing \,|\, A \rightarrow B \,|\, A \rightarrow B, \neg B$}
\RightLabel{\scriptsize{$\wedge$-1}}
\BinaryInfC{$\varnothing \,|\, A \rightarrow B \,|\, A \rightarrow B, A \wedge \neg B$}
\RightLabel{\scriptsize{$\neg$-$0$}}
\UnaryInfC{$\neg (A \wedge \neg B) \,|\, A \rightarrow B \,|\, A \rightarrow B$}
\DisplayProof
\end{center}

\noindent We now show that $A \rightarrow B \vdash_{\sf DF/TT} A \supset B$. First, let $\mathcal{D}_0$ be the following derivation:

\begin{center}
\AxiomC{$\,$}
\RightLabel{\scriptsize{$\mathsf{SRef}$}}
\UnaryInfC{$A, \neg B\,|\, A \,|\, A, A$}
\AxiomC{$\,$}
\RightLabel{\scriptsize{$\mathsf{SRef}$}}
\UnaryInfC{$A, \neg B\,|\, \neg B \,|\, A, \neg B$}
\AxiomC{$\,$}
\RightLabel{\scriptsize{$\mathsf{SRef}$}}
\UnaryInfC{$A, \neg B\,|\, A, \neg B \,|\, A$}
\RightLabel{\scriptsize{$\wedge$-$\nicefrac{1}{2}$}}
\TrinaryInfC{$A, \neg B\,|\, A \wedge \neg B \,|\, A$}
\DisplayProof
\end{center}

\noindent Second, let $\mathcal{D}_1$ be the following derivation: 

\begin{center}
\AxiomC{$\,$}
\RightLabel{\scriptsize{$\mathsf{SRef}$}}
\UnaryInfC{$A, \neg B, B\,|\, A \,|\, A$}
\AxiomC{$\,$}
\RightLabel{\scriptsize{$\mathsf{SRef}$}}
\UnaryInfC{$A, \neg B, B\,|\, \neg B \,|\, \neg B$}
\AxiomC{$\,$}
\RightLabel{\scriptsize{$\mathsf{SRef}$}}
\UnaryInfC{$A, B\,|\, A,  B \,|\, B$}
\RightLabel{\scriptsize{$\neg$-$\nicefrac{1}{2}$}}
\UnaryInfC{$A, B\,|\, A,  \neg B \,|\, B$}
\RightLabel{\scriptsize{$\neg$-$1$}}
\UnaryInfC{$A, \neg B, B\,|\, A,  \neg B \,|\, \varnothing$}
\RightLabel{\scriptsize{$\wedge$-$\nicefrac{1}{2}$}}
\TrinaryInfC{$A, \neg B, B\,|\, A \wedge \neg B \,|\, \varnothing$}
\DisplayProof
\end{center}

\noindent Finally, combining $\mathcal{D}_0$ and $\mathcal{D}_1$ yields the desired result: 

\begin{center}
\AxiomC{$\mathcal{D}_0$}
\noLine
\UnaryInfC{$A, \neg B\,|\, A \wedge \neg B \,|\, A$}
\AxiomC{$\mathcal{D}_1$}
\noLine
\UnaryInfC{$A, \neg B, B\,|\, A \wedge \neg B \,|\, \varnothing$}
\RightLabel{\scriptsize{$\rightarrow$-$0$}}
\BinaryInfC{$A, \neg B, A \rightarrow B \,|\, A \wedge \neg B \,|\, \varnothing$}
\RightLabel{\scriptsize{$\wedge$-$0$}}
\UnaryInfC{$A \wedge \neg B, A \rightarrow B \,|\, A \wedge \neg B \,|\, \varnothing$}
\RightLabel{\scriptsize{$\neg$-$1$}}
\UnaryInfC{$A \rightarrow B \,|\, A \wedge \neg B \,|\, \neg (A \wedge \neg B)$}
\RightLabel{\scriptsize{$\neg$-$\nicefrac{1}{2}$}}
\UnaryInfC{$A \rightarrow B\,|\, \neg (A \wedge \neg B) \,|\, \neg (A \wedge \neg B)$}
\DisplayProof
\end{center}
\end{proof}

%% file: khoo.tex
\section{Gibbardian collapse without Left Logical Equivalence}\label{sec:khoo}

\citet[489]{khoo2019triviality} prove Gibbard's collapse result using Reasoning by Cases. They do not use Left Logical Equivalence (as in our reconstruction of Gibbard's original proof) and explicitly refer to principles (i)--(iii) only (i.e., Import-Export, Stronger-than-Material and Supraclassicality). However, like Gibbard, they actually make use of more assumptions, in particular (v): the classicality of $\supset$. Their proof can be formalized thus:

\begin{thm}
Let $L$ be a reflexive, monotonic, and transitive consequence relation, with $\vee$ satisfying Reasoning by Cases. Then if (i), (ii), (iii) and (v) hold in $L$, $\supset$ entails $\to$, that is, for any $A, B \in {\sf For}(L)$, $A \supset B \models_L A \to B$.
\end{thm}

1. $\neg A \wedge A \models_{\sf CL} B$, classical logic

2. $\models_L (\neg A \wedge A)\to B$, by 1 and (iii)

3. $\models_L \neg A \to (A \to B)$, by 2 and (i)

4. $\neg A \to (A\to B) \models_L \neg A \supset (A \to B)$, by (ii)

5. $\models_L \neg A \supset (A \to B)$, by 3, 4 and Transitivity

6. $\neg A \models_L \neg A \supset (A \to B)$, by 5 and Monotonicity

7. $\neg A \models_L \neg A$, by Reflexivity

8. $\neg A \models_L A\to B$, by 6, 7 and (v), using (meta) Modus Ponens for $\supset$

9. $B \wedge A \models_{\sf CL} B$, classical logic

10. $\models_L (B \wedge A) \to B$, by 9 and (iii)

11. $\models_L B \to (A \to B)$, by 10 and (i)

12. $B \to (A \to B) \models_L B \supset (A \to B)$, by (ii)

13. $\models_L B \supset (A \to B)$, by 11, 12, and Transitivity

14. $B \models_L B \supset (A \to B)$, by 13 and Monotonicity

15. $B \models_L B$ by Reflexivity

16. $B \models_L A\to B$, by 14, 15, (v), using (meta) Modus Ponens

17. $\neg A \vee B\models_L A\to B$, by 8, 16 and Reasoning by Cases

18. $A\supset B\models_L A\to B$, by 17 and (v)\bigskip

This version does not use the replacement principle (iv) of Gibbard's original proof, making it particularly interesting, in particular in relation to {\sf DF/TT}. Indeed, Reasoning by Cases is valid in {\sf DF/TT} and {\sf CC/TT}, as are structural assumptions on logical consequence. We know that {\sf CC/TT} fails Supraclassicality and so step 2 and 10 of the proof are blocked. Interestingly, however, all steps of the proof here are \emph{sound in {\sf DF/TT}}. Although principle (v) does not hold of $\supset$ in full generality in {\sf DF/TT}, all instances of (v) are sound in this case, unlike in Gibbard's original proof. Readers may observe that the proof of $A \supset B \vdash_{\sf DF/TT} A \rightarrow B$ produced in the sequent-system of Appendix \ref{sec:technical} also mirrors Reasoning by Cases (see Lemma \ref{lemma:seq}): on the third line from the root of the tree, the left branch of the derivation tree actually establishes that $\neg A \vdash_{\sf DF/TT} A \rightarrow B$, while the right branch establishes that $B \vdash_{\sf DF/TT} A \rightarrow B$. 
